\documentclass{article}
\usepackage{graphicx} 
\usepackage[margin=25mm]{geometry}
\usepackage{amsmath}
\usepackage{amsfonts}
\usepackage{amssymb}
\usepackage{amsthm}
\usepackage{mathtools} 
\usepackage{enumitem} 
\usepackage{hyperref}
\hypersetup{
    colorlinks,
    linkcolor=black,
    citecolor=black,
    filecolor=black,      
    urlcolor=black,
}

\newtheorem{theorem}{Theorem}[section]
\newtheorem{lemma}[theorem]{Lemma}
\newtheorem{prop}[theorem]{Proposition}

\theoremstyle{definition}
\newtheorem{definition}[theorem]{Definition}
\newtheorem{notation}[theorem]{Notation}

\newtheorem{remark}[theorem]{Remark}

\title{Log Canonical Thresholds for Plane Curves in Arbitrary Characteristic}
\author{Chih-Kuang Lee}
\date{}

\begin{document}

\maketitle

\begin{abstract}
    We generalize the formula for the log canonical threshold(LCT) of plane curves over the complex numbers to arbitrary characteristics. Our proof relies purely on valuation theory, instead of on the theory of $D$-modules.
\end{abstract}

\section{Introduction}
\label{section: Introduction}

Let $\Bbbk$ be an algebraically closed field, $R=\Bbbk[[x,y]]$ the formal power series ring in two variables with maximal ideal $\mathfrak{m}=(x,y)$, $X=\operatorname{Spec}R$ the affine spectrum, $o$ the vanishing locus of $\mathfrak{m}$, and $C$ an analytically irreducible plane curve defined by an element $f\in \mathfrak{m}$. The goal of this paper is to compute the log canonical threshold (LCT) of the pair $(X, C)$. We can assume that the tangent cone of $C$ is the $x$-axis by performing a change of coordinates. In this case, the LCT of the pair is given as follows.

\begin{theorem}
\label{thm: LCT_Formula}
Let ${f}\in\Bbbk[[x,y]]$ be irreducible, and assume that the tangent cone of $V({f})$ is defined by the equation $y^{\operatorname{ord}_\mathfrak{m}({f})} = 0$. Write $v_{{f}}$ for the corresponding non-normalized curve semivaluation. Then, we have
\[
\operatorname{lct}_o({f}) = \frac{1}{v_{{f}}(x)}+\frac{1}{v_{f}(y)}.
\]
\end{theorem}

Our result generalizes the statement in \cite[Example 8.9]{Kollar1997}, which says that if $f\in \mathbb{C}[[x,y]]$ is irreducible, then
\[
    \operatorname{lct}_o(f) = \frac{1}{m} + \frac{1}{n},
\]
where $m = \operatorname{ord}_\mathfrak{m}(f)$, and $n/m$ is the first Puiseux exponent of $f$.     
This statement is proven in \cite{Igusa1977}. Originally, Igusa deals with different problems and computes invariants related to Bernstein-Sato polynomials, which turn out to be related to the LCT. Our approach uses valuation theory directly, instead of $D$-module theory.

In Section \ref{section: Valuative_Tree}, we build the terminology in the theorem, and the techniques for a part of the proof of the theorem. In particular, we review the valuation theory of the plane using the same approach as \cite{FavreJonsson2004}, but working in arbitrary characteristics. Another part of the proof relies on the study of the Newton polyhedron of the monomial ideal induced by $f$, which will be discussed in Section \ref{section: Newton_Polytope}. Finally, we prove our main theorem in Section \ref{section: Main_Theorem}.

Theorem~\ref{thm: Isom_between_Gamma_and_Vm} is a critical ingredient in our developments.
In \cite{FavreJonsson2004}, a proof that only works in characteristic $0$ is provided. We offer an alternative proof, which works in arbitrary characteristic, at the end of Subsection~\ref{subsection: The Universal Dual Graph}.


\section{Valuation Theory}
\label{section: Valuative_Tree}

In this section, we set up notations and terminology related to valuations; they are mostly from \cite{FavreJonsson2004}.

\subsection{Semivaluations}
A (nonnegative real) \emph{semivaluation} on a ring $A$ is a function $v\colon A\to [0,\infty]$ satisfying valuation hypotheses, but only requiring that $v(0) = \infty$ (that is, we allow $v(f) = \infty$ for nonzero $f\in A$).
We keep the terminology \emph{valuation} for those $v$ for which $v(f) = \infty$ implies $f=0$, and for classical valuations defined on a field.
A semivaluation $v$ is \emph{centered} if $v(f)>0$ for some nonzero $f\in A$; that is, the \emph{center} $c_A(v) = v^{-1}((0,\infty])$ is nonzero. For an ideal $I\subset A$, we write $v(I) = \inf\{v(f)\mid f\in A\}$.

On a Noetherian separated scheme $X$, a \emph{semivaluation} is a pair $(v,x)$, where $x\in X$ is a (non-necessarily closed) point, and $v$ is a valuation on the residue field $\kappa(x)$ which is \emph{centered} on $X$ in the following sense: for an open neighborhood $\operatorname{Spec}A$ of $x$, the valuation $v$ gives a centered semivaluation on $A$. This definition is independent of the choice of the neighborhood $\operatorname{Spec}A$, and we denote by $c_X(v)$ the point in $X$ corresponding to the prime ideal $c_A(v)$ of $A$. 

We describe some semivaluations on the ring $R \coloneq \Bbbk[[x_1, x_2,\dots,x_n]]$.
A monomial $x_1^{u_1}\cdots x_n^{u_n}$ in $R$ will be written as $x^\emph{u}$ for $\emph{u} = (u_1,\dots,u_n)\in\mathbb{Z}_{\geq 0}^n$.
For $\mathbf{w}\in \mathbb{R}_{\geq0}^n$, we define the \emph{monomial valuation} $v_{\mathbf{w}}$ on $R$ as
\[
v_{\mathbf{w}}\left(\sum_{\operatorname{u}}c_{\operatorname{u}} x^{\operatorname{u}}\right)
=\min\left\{ \mathbf{w}\cdot\operatorname{u} \,|\, c_{\operatorname{u}}\neq 0\right\}.
\]
Notice that $v_{\mathbf{w}}$ is centered if and only if $\mathbf{w}\in\mathbb{R}_{>0}^n$.

Let $\pi\colon Y\to X \coloneq\operatorname{Spec}R$ be a (proper) birational morphism from a normal variety $Y$, and let $E\subset Y$ be a prime divisor which gives the pullback on functions $\pi^\#\colon R\to \mathcal{O}_{Y, E}$. Then the composition
\[
\pi_*\operatorname{ord}_E =\operatorname{ord}_E \mathrel{\circ} \pi^{\#},
\]
followed by a scaling by a positive number $b_E$, defines a valuation, called a \emph{divisorial valuation}.
With slight abuse the notation, we write $\pi_*\operatorname{ord}_E$ simply as $\operatorname{ord}_E$.
Pick a closed point $y\in Y$ and a monomial valuation $v_{\mathbf{w}}$ with respect to some local system of parameters at $y$. Then the composition
\[
\pi_*v_{\mathbf{w}} = v_{\mathbf{w}}\circ\pi^{\#},
\]
defines a valuation in $\mathcal{V}$, named a \emph{quasimonomial valuation}. Note that every divisorial valuation is quasimonomial; a quasimonomial valuation which is not divisorial is called \emph{irrational}.

In the following example, let $n=2$ for simplicity. Let $f\in R$ be an irreducible element, which gives an analytically irreducible curve $\operatorname{V}(f)$ on $X$. The intersection multiplicity defined by
\[
v_f(g) = \dim_\Bbbk\frac{\Bbbk[[x_1,x_2]]}{(f,g)}
\]
is a semivaluation, called a \emph{curve semivaluation}. Note that $v_f(g) = \infty$ if and only if $g$ is a multiple of $f$, and thus if $v_f = v_g$ then $f=gu$ for some $u\in R^\times$. 

\subsection{Log Canonical Thresholds (LCTs)}
Let $Z$ be a closed subscheme of an integral scheme $X$ with defining ideal sheaf $\mathcal{I}$, and let $E$ be a divisor over $X$.
We define a number $\operatorname{ord}_E(Z)$ as follows: algebraically, we can do this locally by letting $X = \operatorname{Spec}A$ for some integral domain $A$ and $Z = \operatorname{Spec} A/\mathfrak{a}$ for some ideal $\mathfrak{a}\subset A$.
Write $K \coloneq \mathcal{K}(X) = \operatorname{Frac} A$. The valuation $\operatorname{ord}_E$ induces a DVR $V\subset K$ containing $A$, and we set $\operatorname{ord}_E(Z) = \operatorname{ord}_E(\mathfrak{a}V)$.
This gives the following geometric interpretation of $\operatorname{ord}_E(Z)$. There is a birational morphism $\pi\colon Y\to X$ with $Y$ normal and containing $E$.
Then, $\operatorname{ord}_E(Z)$ is the vanishing order of the ideal sheaf $\mathcal{I}\cdot\mathcal{O}_Y$ along $E$.

Now, assume that $X$ is \emph{$\mathbb{Q}$-Gorenstein}; that is, $X$ is normal, and if $K_X$ is a canonical divisor of $X$, then $mK_X$ is Cartier for some positive integer $m$.
For any birational morphism $\pi\colon Y\to X$ with $Y$ normal and containing a prime divisor $E$, the \emph{log discrepancy} of $E$ over $X$ is $A_X(E)\coloneq \operatorname{ord}_E(K_{Y/X}) + 1$, where $K_{Y/X}$ is the relative canonical divisor of $Y$ over $X$. This definition is independent of the model $Y$ where $E$ lives.

We further assume that $X$ is log terminal (i.e., the pair $(X,\emptyset)$ is log terminal), and let $W$ be a closed subscheme of $X$.
The \emph{log canonical threshold} (LCT) of the pair $(X,Z)$ along $W$ is 
\[
\operatorname{lct}_W(X,Z) \coloneq \inf \left\{ \frac{A_X(E)}{\operatorname{ord}_E(Z)}\,\middle|\, E\mbox{ is a divisor over } X \mbox{ such that } c_X(\operatorname{ord}_E) \in W \right\}.
\]
A \emph{log resolution} of the pair $(X,Z)$, where $X$ is $\mathbb{Q}$-Gorenstein, is a proper birational morphism $\pi\colon Y\to X$ with $Y$ regular, such that $\mathcal{I}\cdot\mathcal{O}_Y = \mathcal{O}_Y(-D)$ for some Cartier divisor $D$, and that $K_{Y/X}+D$ is simple normal crossings.
If $\pi\colon Y\to X$ is a log resolution of $(X, Z)$, then the infimum of the LCT can be taken among those $E\subset Y$ with $c_X(\operatorname{ord}_E)\in W$.
The LCT is independent of the log resolution chosen.

\subsection{Non-metric Trees}
\label{subsection: Non-metric_Trees}
    Let $\Lambda$ be a totally ordered set. A \emph{non-metric $\Lambda$-tree}\footnote{Our notion is called a \emph{rooted nonmetric $\Lambda$-tree} in \cite{FavreJonsson2004}, and by `forgetting the root' we get their version of nonmetric tree.} is a poset $(\mathcal{T},\leq)$ with the following properties.
    \begin{enumerate}[label={(T\arabic*)}]
        \item \label{def:Nonmetric_Tree:infimum} Every nonempty subset $\mathcal{S}$ of $\mathcal{T}$ admits an infimum $\wedge_{\tau\in \mathcal{S}}\tau$.
        The unique minimal element $\tau_0$ of $\mathcal{T}$ is called the \emph{root} of $\mathcal{T}$.
        \item \label{def:Nonmetric_Tree:left_ray} If $\tau \in\mathcal{T}$, the set $\{ \sigma \,|\, \sigma \leq \tau\}$ is isomorphic to a $\Lambda$-interval.
        \item \label{def:Nonmetric_Tree:not_long_line} Every convex subset of $\mathcal{T}$ is isomorphic to a $\Lambda$-interval. 
    \end{enumerate}
    We will use the term \emph{non-metric tree} to refer to a non-metric $\mathbb{R}$-tree, and denote $(T,\leq)$ simply by $\mathcal{T}$ if the order is understood. A non-metric $\Lambda$-tree is \emph{complete} if every totally ordered subset of $\mathcal{T}$ has an upper bound. Also, a \emph{non-rooted non-metric $\Lambda$-tree} is obtained from a non-metric $\Lambda$-tree by forgetting the root.

    A \emph{parameterization} on a non-metric $\Lambda$-tree $(\mathcal{T},\leq)$ is an increasing function $\alpha\colon \mathcal{T}\to \Lambda\cup\{\pm \infty\}$ such that it gives a bijection from any convex subset of $\mathcal{T}$ to an interval in $\Lambda$.
    In the case that $\Lambda=\mathbb{R}$, we also allow the codomain of $\alpha$ to be $[0,\infty]$ or $[0,1]$.
    A \emph{parametrized $\Lambda$-tree} is a non-metric $\Lambda$-tree with a parametrization.

    Let $\mathcal{T}$ be a non-metric $\Lambda$-tree. Elements $\tau,\tau'\in \mathcal{T}$ define a \emph{segment} by
        \[
        [\tau,\tau'] \coloneq \{ \sigma \mid \tau\wedge\tau' \leq \sigma \leq \tau \mbox{ or } \tau\wedge\tau' \leq \sigma \leq \tau'\}.
        \]
    We define an open-closed interval $(\tau,\tau']$ or an open interval $(\tau,\tau')$ in the same way without the corresponding endpoints.

    Given non-metric $\Lambda$-trees $\mathcal{S},\mathcal{T}$ with roots $\sigma_0,\tau_0$ respectively, a function $\Phi\colon \mathcal{S}\to \mathcal{T}$ is a \emph{morphism} of non-metric $\Lambda$-trees if for any $\sigma\in\mathcal{S}$, the map $\Phi$ is an order-preserving bijection on segments $[\sigma_0,\sigma]$ and $[\tau_0,\Phi(\sigma)]$.
    If $\Phi$ is also bijective, then it is an \emph{isomorphism} of non-metric $\Lambda$-trees.
    When $\mathcal{S},\mathcal{T}$ are parametrized trees with parametrizations $\alpha,\beta$ respectively, $\Phi$ is an \emph{isomorphism} of parametrized trees if it is an isomorphism of non-metric trees and $\alpha = \beta\circ\Phi$.
    

    Let $\tau$ be an element of a non-metric $\Lambda$-tree $\mathcal{T}$. Define a relation on $\mathcal{T}\setminus \{\tau\}$ by declaring $\sigma$ and $\sigma'$ to be equivalent if the segments $(\tau,\sigma]$ and $(\tau,\sigma']$ intersect. This is an equivalence relation, and an equivalence class $[\sigma]$ is called a \emph{tangent vector} at $\tau$. The collection of all equivalence classes is called the \emph{tangent space} at $\tau$, denoted by $T_\tau$. If $\sigma\in\mathcal{T}$ satisfies $[\sigma]=\mathbf{v}\in T_\tau$, we say that \emph{$\sigma$ represents $\mathbf{v}$}.
    
    Let $(\mathcal{T},\leq)$ be a non-metric tree. Given a point $\tau\in\mathcal{T}$ and a tangent vector $\mathbf{v}\in T_\tau$, define
    \[
    U_\tau(\mathbf{v}) = \{ \sigma\in\mathcal{T}\setminus\{\tau\} \mid \sigma \mbox{ represents }\mathbf{v} \}.
    \]
    The \emph{observer’s topology}\footnote{The terminology observer's topology is from the paper \cite{Coulbois2007}. It is called \emph{the weak tree topology} in \cite{FavreJonsson2004}.} is the topology generated by the subbase given by the collection of all possible $U_\tau(\mathbf{v})$.

\begin{prop}[{\cite[Proposition 3.8]{FavreJonsson2004}}]
\label{prop:parametrication_is_w.l.semicont}
    Any parametrization $\alpha$ on a non-metric tree $(\mathcal{T},\leq)$ is lower semi-continuous for the observer's topology.
\end{prop}

\subsection{The Valuative Tree \texorpdfstring{$\mathcal{V}$}{V}}
\label{subsection: The_Valuative_Tree_V}
\begin{notation}
    Throughout the rest of this paper, we fix the notations $R \coloneq \Bbbk[[x,y]]$ and $\mathfrak{m} \coloneq (x,y)$.
\end{notation}
We denote by $\mathcal{V}$ the collection of all centered semivaluations $v$ on $R$ such that $v(\mathfrak{m}) \coloneq \min\{v(x),v(y)\} = 1$. For $v,w\in\mathcal{V}$, we define $v\leq w$ if $v(f)\leq w(f)$ for all $f\in R$, which makes $(\mathcal{V,\leq})$ a poset. One can show that $(\mathcal{V},\leq)$ is a non-metric tree rooted at $\operatorname{ord}_\mathfrak{m} = v_{(1,1)}$ (see \cite[Theorem~3.14]{FavreJonsson2004} for a proof). In particular, any two elements $v,w\in\mathcal{V}$ admit a greatest lower bound $v\wedge w$.

Since $\mathcal{V}$ is a non-metric tree, it is endowed with the observer's topology, which coincides with the \emph{weak topology} on $\mathcal{V}$, the weakest topology making all evaluation maps $v\mapsto v(f)$ continuous for $f\in R$.

Note that the semivaluations $v$ given in the examples of Subsection~\ref{subsection: Non-metric_Trees} yield normalized semivaluations $v(\mathfrak{m})^{-1}v\in\mathcal{V}$.
However, these do not exhaust the whole space $\mathcal{V}$; there remain other valuations, called \emph{infinitely singular valuations}.
Since scaling a semivaluation by a nonnegative real number produces an equivalent one, the type of a semivaluation $c\cdot v$, where $c>0$, is by definition the same as the type of $v$.

\subsection{Sequences of Key Polynomials (SKPs)}
To analyze the valuative tree $\mathcal{V}$, we review the notion of a sequence of key polynomials, which Mac Lane initially developed in \cite{MacLane1936}. We follow and borrow the symbols from \cite{FavreJonsson2004}.
Note that \cite{FavreJonsson2004} only deals with the case $\Bbbk = \mathbb{C}$, but the arguments in that book work for any algebraically closed field.

    Let $k\in\overline{\mathbb{Z}}_{>0}$. A \emph{sequence of key polynomials} (SKP) of $R$ is a sequence of polynomials $(U_j)_{j=0}^k$ and a sequence $(\tilde\beta_j)_{j=0}^k$ of numbers in $(0,\infty]$, not all infinity, that satisfy the following conditions:
    \begin{enumerate}[label={(S\arabic*)}]
        \item $U_0=x$ and $U_1 = y$.
        \item \label{(S2):SKP_relations} For $1\leq j< k$, there are relations
        \[
        \tilde\beta_{j+1} > n_j\tilde\beta_{j} = \sum_{l=0}^{j-1} m_{j,l}\tilde\beta_l,
        \]
        where $n_j\in\mathbb{Z}_{>0}$ and $ m_{j,l}\in\mathbb{Z}_{\geq 0}$ satisfy the conditions: for all $1\leq j < k$ and $0 < l< j$,
        \[
        n_j = \min\left\{ n\in\mathbb{Z}_{>0}\,\middle|\, n\tilde\beta_j\in \mathbb{Z}\tilde\beta_1+\cdots+ \mathbb{Z}\tilde\beta_{j-1}\right\}
        \]
        and $0\leq m_{j,l}<n_j$.
        \item \label{(S3):Next_SKP} For $j\geq 1$, there is $\theta_j\in \Bbbk^*$ such that
        \[
        U_{j+1} = U_j^{n_j} -\theta_jU_0^{m_{j,0}}U_1^{m_{j,1}}\cdots U_{j-1}^{m_{j,j-1}}.
        \]
    \end{enumerate}
    An SKP can be denoted by a pair $[(U_j)_0^k;(\tilde\beta_j)_0^k]$. If $k<\infty$, we say that the SKP is \emph{finite}, and \emph{infinite} otherwise.

The following lemma will be used in the last section.
\begin{lemma}[First Properties of SKPs]
    \label{lemma: First_Prop_of_SKP}
    Let $[(U_j)_0^k;(\tilde\beta_j)_0^k]$ be an SKP.
    \begin{itemize}
        \item All $U_j$ are irreducible in $R$ and in Weierstrass form with respect to $y$.
        \item The sequence $(\tilde\beta_j/\deg_y U_j)_0^k$ is strictly increasing.\footnote{We may define $\deg_y U_0 = \deg_y x = -\infty$, so $\tilde\beta_0/\deg_y U_0 = 0$.}
    \end{itemize}
\end{lemma}

An SKP $[(U_j)_0^k;(\tilde\beta_j)_0^k]$, where $1\leq k \leq \infty$, determines a smallest centered semivaluation $v=\operatorname{val}[(U_j)_0^k;(\tilde\beta_j)_0^k]$ such that $v(U_j) = \tilde\beta_j$ for all $0\leq j \leq k$; conversely, any centered semivaluations induces an SKP.

Let us sketch the idea of how an SKP determines a semivaluation by induction on the length $k$ of an SKP. We refer to \cite{FavreJonsson2004} for a complete proof. Assume $k<\infty$ for a moment. When $k=1$, define the monomial valuation $v_1$ with the weight $(\tilde\beta_0,\tilde\beta_1)$. Assume that $v_1,\dots,v_{k-1}$ are defined from $[(U_j)_0^{k-1};(\tilde\beta_j)_0^{k-1}]$. Note that for $f\in \Bbbk[x,y]$, we can write
\[
f = \sum_i g_iU_{k}^i,
\]
where $g_i\in \Bbbk[x,y]$ with $\deg_y g_i<\deg_yU_{k}$. Define
\[
v_{k}(f) = \min_i \left\{ v_{k-1}(g_i) + i\tilde\beta_{k} \right\}.
\]
Then, $v_k$ is a semivaluation on $\Bbbk[x,y]$ (see \cite[Section 2.1.3]{FavreJonsson2004}), and it extends uniquely to a semivaluation on $R$ (see, for example, \cite[Proposition~2.10]{FavreJonsson2004}). We then let $v = v_k$.
Finally, the case of $k=\infty$ follows from the following two facts (see \cite[Theorem~2.22]{FavreJonsson2004}).
\begin{itemize}
    \item $v_k$ converges (weakly) to a semivaluation $v_\infty$ as $k\to\infty$.
    \item If $n_j = 1$ for all $k$ large enough, then $U_k$ converges to an irreducible formal power series $U_\infty\in R$, and $v_\infty = v_{U_\infty}$.
\end{itemize}

One can completely classify centered semivaluations on $R$ via SKPs (see \cite[Definition~2.23]{FavreJonsson2004}).

\subsection{Parametrizations on \texorpdfstring{$\mathcal{V}$}{V}}
\label{subsection: Parametrizations_on_V}
We see from Subsection~\ref{subsection: The_Valuative_Tree_V} that $\mathcal{V}$ is a non-metric tree; here, we review some natural parametrizations on $\mathcal{V}$ by first looking at some numerical invariants.
    Given a semivaluation $v\in\mathcal{V}$, its \emph{skewness} is
    \[
    \alpha(v) \coloneq \sup \left\{ \frac{v(f)}{\operatorname{ord}_\mathfrak{m}(f)} \,\middle|\, f\in \mathfrak{m} \right\},
    \]
    and its \emph{multiplicity}\footnote{In \cite{FavreJonsson2004}, this definition only applies for quasimonomial valuations, but it can be extended to any semivaluation.} is
    \(
    m(v) \coloneq \operatorname{min}\left\{  \operatorname{ord}_\mathfrak{m}(f) \mid v_f\geq v \right\}
    \).

\begin{prop}[{cf. \cite[Proposition 3.25]{FavreJonsson2004}}]
\label{prop:Using_Skewness_to_Complute_Semivaluation}
    Given $v\in\mathcal{V}$ and any irreducible $f\in\mathfrak{m}$, we have
    \[
    v(f) = \alpha(v\wedge v_f) \operatorname{ord}_\mathfrak{m}(f).
    \]
    This immediately implies that $\alpha(v) = v(f)/ \operatorname{ord}_\mathfrak{m}(f)$ if and only if $v_f \geq v$.
\end{prop}
SKPs give explicit formulas for the skewness and the multiplicity.

\begin{prop}
\label{prop: SKP_Computes_Skewness_and_Multiplicity}
If $v=\operatorname{val}[(U_j)_0^k;(\tilde\beta_j)_0^k]$, then we can read the skewness directly (see \cite[Lemma~3.32]{FavreJonsson2004}):
\begin{itemize}
    \item If $v$ is quasimonomial, then $\alpha(v) = \tilde\beta_0\tilde\beta_k/\deg_yU_j$.
    \item If $v$ is a curve semivaluation, then $\alpha(v) = \infty$.
    \item If $v$ is infinitely singular, then $\alpha(v) = \lim_{j\to\infty}\tilde\beta_0\tilde\beta_j/\deg_yU_j\in (1,\infty]$.
\end{itemize}

If we further assume that $1 = \tilde\beta_0 \leq \tilde\beta_1$, then $m(v) = \sup \left\{ \deg_y U_j \mid 1\leq j \leq k\right\}$. In particular, if $v$ is a quasimonomial valuation or a curve semivaluation, then $m(v) = \deg_y U_k$; $v$ is infinitely singular if and only if $m(v) = \infty$ (see \cite[Lemma~3.42 and Proposition~3.37]{FavreJonsson2004}).
\end{prop}


    
    
        For $v\in\mathcal{V}$ with $m(v)<\infty$, there exists a finite sequence of divisorial valuations $v_1,v_2,\dots,v_g$ so that
        \[
         \operatorname{ord}_\mathfrak{m} = v_0 < v_1 < v_2 < \cdots < v_g < v_{g+1} = v,
        \]
        and a strictly increasing sequence of positive integers $m_0,m_1,m_2,\dots,m_g,m_{g+1}$ such that $m(w) = m_j$ for all $w\in (v_{j-1},v_j]$ and $1\leq j \leq g+1$.
        A pair of such sequences $[(v_j)_0^g;(m_j)_0^g]$ is an \emph{approximating sequence}\footnote{In \cite{FavreJonsson2004}, they denote an approximating sequence only using $(v_j)_0^g$. Here, we also include the data $(m_j)_0^g$.} for $v$.
        If $v$ is infinitely singular, we can extend the previous definition naturally to this case (with $g=\infty$).
    

    The \emph{thinness}\footnote{For a motivation of the name `thinness', see \cite[Remark~3.49]{FavreJonsson2004}.} of a valuation $v\in\mathcal{V}$ is
    \[
    A^{\textup{thin}}(v) \coloneq 2+\int_{ \operatorname{ord}_\mathfrak{m}}^v m(w)d\alpha(w).
    \]
    Equivalently, if $[(v_j)_0^g;(m_j)_0^g]$ is the approximating sequence of $v$, then
    \[
    A^{\textup{thin}}(v) = 2 + \sum_{j=0}^{g} m_j(\alpha(v_{j+1}) - \alpha(v_j)).
    \]

Finally, the skewness and the thinness are lower semicontinuous with respect to the observer's topology. This is because they are parametrizations on $\mathcal{V}$ (see \cite[Theorem~3.26]{FavreJonsson2004}), and a parametrization on a non-metric tree is always lower semicontinuous (see \cite[Proposition 3.8]{FavreJonsson2004}).

\subsection{The Universal Dual Graph}
\label{subsection: The Universal Dual Graph}
We review another approach to the valuative tree, which is from the point of view of birational geometry and is more combinatorial. We mostly follow \cite[Chapter~6]{FavreJonsson2004}.

We define $\mathfrak{B}$ to be the collection of all proper birational maps $\pi\colon Y\to \operatorname{Spec}R$ with $Y$ a regular scheme;
we also note that $\pi$ is always a sequence of blow-ups at closed points (see, for example, \cite[\href{https://stacks.math.columbia.edu/tag/0C5R}{Tag 0C5R}]{stacks-project}).
Each $\pi\in\mathfrak{B}$ is associated with a (combinatorial) graph $\Gamma_\pi$, called the \emph{dual graph} of $\pi$, constructed as follows: Vertices of $\Gamma_\pi$ are irreducible components of $\pi^{-1}(\mathfrak{m})$, and if two irreducible components intersect, there is an edge between them. We observe the following:
\begin{itemize}
    \item The collection of vertices of $\Gamma_\pi$ is denoted by $\Gamma_\pi^*$, and is also called the \emph{dual graph} of $\pi$ if no confusion arises.
    \item If $\pi_0:\operatorname{Bl}_\mathfrak{m}R \to \operatorname{Spec}R$ is the single blow-up at $\mathfrak{m}$, we write $E_0$ the exceptional divisor of $\pi_0$.
    By the universal property of the blow-up, $E_0\in\Gamma_\pi$ for all $\pi\in\mathfrak{B}$.
    \item For $\pi\in\mathfrak{B}$, define a partial order on $\Gamma_\pi^*$ as follows:
    $E'\leq_\pi E$ if there is a containment of segments $[E_0, E']\subset [E_0, E]$ in $\Gamma_\pi$.\footnote{A segment $[v, v']$ in a (combinatorial) tree is the minimal subtree containing the vertices $v$ and $v'$.}
    This construction makes $(\Gamma_\pi^*,\leq_\pi)$ a poset.
\end{itemize}
Define a partial order on $\mathfrak{B}$ by declaring that $\pi'\leq \pi$ if $\pi$ factors through $\pi'$, which makes $(\mathfrak{B},\leq)$ a directed poset.
Note that if $\pi'\leq \pi$ then there is an order-preserving inclusion map between dual graphs $\iota_{\pi',\pi}\colon\Gamma_{\pi'}^*\to\Gamma_\pi^*$, which indicates that
$((\Gamma_\pi^*,\leq_{\pi})_{\pi\in\mathfrak{B}},(\iota_{\pi',\pi})_{\pi'\leq\pi\in\mathfrak{B}})$ is a direct system of posets. The colimit (direct limit) of this direct system
\[
(\Gamma^*,\leq) \coloneq \operatorname*{colimit}_{\pi\in\mathfrak{B}} (\Gamma_\pi^*,\leq_\pi)
\]
is called the \emph{universal dual graph}.
Intuitively, $\Gamma^*$ is the union of all dual graphs $\Gamma_\pi^*$.
We also see that $E'\leq E$ in $\Gamma^*$ if and only if $E'\leq_\pi E$ for some $\pi\in\mathfrak{B}$.
The universal dual graph $(\Gamma^*,\leq)$ is a non-metric $\mathbb{Q}$-tree, rooted at $E_0$. $\Gamma^*$ induces a smallest complete non-metric $\mathbb{R}$-tree containing $\Gamma^*$, which is also called the \emph{universal dual graph}, denoted by $\Gamma$.

We recall some important numerical invariants on the universal dual graph. The \emph{Farey Weight} is a function $\Gamma^*\to\mathbb{Z}_{\geq 1}^2$ sending $E$ to a pair of positive integers $(a(E),b(E))$, constructed as follows. Recall that every $E\in\Gamma^*$ can be reached by a sequence of point blow-ups (see \cite[Exercise~9.9]{SwansonHuneke2006}), so we can define the Farey Weight inductively.
First, define $(a(E_0),b(E_0)) = (2,1)$. Now let $E\in\Gamma^*$; then, there is $\pi\colon Y\to\operatorname{Spec}R$ such that $E$ is the exceptional divisor of a point blow-up $\beta:\operatorname{Bl}_pY\to Y$, where $p$ is some point in $\pi^{-1}(\mathfrak{m})$. If $p$ is only in a single divisor $E'$ of $Y$, define $(a(E),b(E)) = (a(E')+1,b(E'))$; otherwise, if $p$ lies in the intersection of two exceptional divisors $E'$ and $E''$ of $Y$, define $(a(E),b(E)) = (a(E'),b(E')) + (a(E''),b(E''))$. Finally, for $E\in\Gamma^*$, we define
\[
m^{\textup{Farey}}(E) = \min\{ b(E') \mid E'\geq E \},
\]
called the \emph{Farey multiplicity} of $E$.\footnote{In \cite{FavreJonsson2004}, they simply call $m^{\textup{Farey}}$ multiplicity and use the notation $m$.}

The relationship between $\mathcal{V}$ and $\Gamma$ is stated in the following theorem.
\begin{theorem}[The Isomorphism between $\Gamma$ and $\mathcal{V}$]
    \label{thm: Isom_between_Gamma_and_Vm}
    There is an isomorphism
    \[    
        \Phi\colon(\Gamma,A^{\textup{Farey}})\to(\mathcal{V},A^{\textup{thin}})
    \]
    of parametrized trees preserving the multiplicities (i.e., $m^{\textup{Farey}} = m\circ\Phi$).
\end{theorem}

\cite[Section~6.4]{FavreJonsson2004} provides a proof of Theorem~\ref{thm: Isom_between_Gamma_and_Vm} when $\operatorname{char}\Bbbk = 0$. However, in that book after Lemma 6.27, they prove that $A^{\textup{Farey}} = A^{\textup{thin}}\circ\Phi$ by using the statement that $A^{\textup{Farey}} = 1+\hat{\beta}$, which is true in characteristic 0 but false in positive characteristic (here, $\hat\beta$ is the Puiseux parameter, the negative log of Berkovich radius, on the closure of an open disc of the Berkovich space of $\mathbb{A}^1_K$, where $K$ is the completion of Puiseux series field). The good news is that Theorem \ref{thm: Isom_between_Gamma_and_Vm} remains true in positive characteristic, but it requires a different proof.

In the rest of this section, we establish an alternative proof of Theorem~\ref{thm: Isom_between_Gamma_and_Vm} that works in arbitrary characteristic.
The strategy is first to prove it on the restriction $\Gamma^*$; for this purpose, we list the properties needed.
Then, we reprove the statement $A^{\textup{Farey}} = A^{\textup{thin}}\circ\Phi$ in Theorem~\ref{thm: A_Farey_Equals_A_Thinness}.

\begin{definition}[$\mathcal{V}_{\textup{div}}$]
    We define $\mathcal{V}_{\textup{div}}$ to be the collection of all divisorial valuations in $\mathcal{V}$.
\end{definition}

\begin{lemma}
    \label{lemma: Isom_between_Gamma_*_and_V_div}
    The natural map $\Phi\colon\Gamma^*\to\mathcal{V}_{\textup{div}}$ sending $E$ to its normalized divisorial valuation $v_E$ is an order-preserving bijection.
\end{lemma}

\begin{proof}
    \cite[Subsection~6.5.1 and Subsection~6.5.3]{FavreJonsson2004} provides a proof in characteristic 0 that also works in positive characteristic.
\end{proof}

\begin{lemma}
    \label{lemma: Geometric_Meaning_of_Farey_Parameters}
    Given $\pi\in\mathfrak{B}$ and $E\in\Gamma^*$, we have equalities $a(E) = A_X(E)$ and $b(E) = \operatorname{ord}_E(\pi^*\mathfrak{m})$,
    where $X = \operatorname{Spec}R$.
\end{lemma}

\begin{proof}
    Write $\pi\colon Y\to X$ as a sequence of $n+1$ blow-ups of closed points.
    The only relevant case is that $E$ is the exceptional divisor in the last blow-up, so we assume it and write $\pi = \pi'\circ\varphi$, where $\pi'\colon Y'\to X$ is the first $n$ blow-ups and $\varphi\colon Y\to Y'$ is the last blow-up at a closed point $p\in (\pi')^{-1}(\mathfrak{m})$.
    We prove by induction on $n$; the base case $n=0$ is clear.

    To show that $a(E) = A_X(E)$, note from \cite[Appendix~A]{deFernex&Ein&Mustata2011} that the relative canonical divisors can be defined via \emph{the sheaf of special differentials}, and we have the transformation rule
    \[
    K_{Y/X} = K_{Y/Y'} + \varphi^*\left(K_{Y'/X}\right)
    \]
    from \cite[Lemma~A.13]{deFernex&Ein&Mustata2011}. If $p$ is in a single exceptional divisor $E'$ of $Y'$, we have
    \begin{align*}
    A_X(E)
    &= 1 + \operatorname{ord}_E\left(K_{Y/X}\right) \\
    &= 1 + \operatorname{ord}_E(K_{Y/Y'}) + \operatorname{ord}_E(\varphi^*(K_{Y'/X}))\\
    &= 1 + 1 + \operatorname{ord}_{E'}(K_{Y'/X}) \\
    &= 1 + a(E') \\
    &= a(E).
    \end{align*}
    Similarly, if $p$ is in the intersection of exceptional components $E'$ and $E''$ of $Y'$, we obtain
    \begin{align*}
    A_X(E)
    &= 1 + \operatorname{ord}_E(K_{Y/Y'}) + \operatorname{ord}_E(\varphi^*(K_{Y'/X}))\\
    &= 1 + 1 + \operatorname{ord}_{E'}(K_{Y'/X}) + \operatorname{ord}_{E''}(K_{Y'/X}) \\
    &= a(E') + a(E'') \\
    &= a(E).
    \end{align*}

    Checking $b(E) = \operatorname{ord}_E(\pi^*\mathfrak{m})$ is straightforward. If $p$ is in a single divisor $E'$ locally cut out by a function $f$ and if $E$ is locally cut by a function $g$, then
    \[
        \pi^*\mathfrak{m}
        = \varphi^*\pi'^*\mathfrak{m}
        = \varphi^*\left(g^{b(E')}u'\right)
        = f^{b(E')}u\varphi^*(u'),
    \]
    where $u'$ and $u$ are local invertible functions around $E'$ and $E$ respectively. Similarly, if $p$ is in the intersection of two exceptional divisors $E'$ and $E''$, locally cut out by functions $g_1$ and $g_2$, then
    \[
        \pi^*\mathfrak{m}
        = \varphi^*\pi'^*\mathfrak{m}
        = \varphi^*\left(g_1^{b(E')}g_2^{b(E'')}u'\right)
        = f^{b(E')+b(E'')}u\varphi^*(u'),
    \]
    where again $u'$ and $u$ are local invertible functions around $E'$ and $E$ respectively.
\end{proof}

\begin{lemma}
    \label{lemma:Multipliticy_is_Constant_on_Segments}
    Let $E, E'\in\Gamma^*$ so that $E'$ is the exceptional divisor of a blow-up at a point lying only on a single exceptional divisor $E$. We then have $m^{\textup{Farey}}(E') = b(E') = b(E)$, and the Farey multiplicity is constant on the interval $(E, E']$ with value $b(E')$.
\end{lemma}

\begin{proof}
    \cite[Subsection~6.16]{FavreJonsson2004} provides a proof in characteristic 0 that also works in positive characteristic.
\end{proof}

\begin{lemma}
    \label{lemma:Farey_Multiplicity_Equals_Usual_Multiplicity}
    The map $\Phi\colon\Gamma^*\to\mathcal{V}_{\textup{div}}$ satisfies $m^{\textup{Farey}} = m\circ\Phi$.
\end{lemma}

\begin{proof}
    The details of the proof can be found in \cite{FavreJonsson2004}, which works in arbitrary characteristic. Let us go through the rough idea.

    The first step is to identify all analytically irreducible curves $C$ as ends of $\Gamma$ via \cite[Proposition~6.12]{FavreJonsson2004}. The Farey multiplicity $m^{\textup{Farey}}$ naturally extends to those curves by taking limits.\footnote{In \cite{FavreJonsson2004}, $m^{\textup{Farey}}(C)$ is denoted by $m_\Gamma(C)$.} From \cite[Corollary~6.21]{FavreJonsson2004}, we obtain
    \[
    m^{\textup{Farey}}(E) = \min\left\{m^{\textup{Farey}}(C)\mid C > E\right\}.
    \]
    Then, \cite[Lemma~6.30]{FavreJonsson2004} says that $m^{\textup{Farey}}(C) = m(v_f)$ whenever $C = V(f)$. For $E\in\Gamma^*$, we can conclude that
    \begin{align*}
        m(\Phi(E))
        &= \min\left\{ m(v_f)\mid v_f > \Phi(E) \right\} \\
        &= \min\left\{ m^{\textup{Farey}}(C)\mid C > E \right\} \\
        &= m^{\textup{Farey}}(E)
    \end{align*}
    as desired.
\end{proof}

\begin{theorem}
    \label{thm: A_Farey_Equals_A_Thinness}
    The natural map $\Phi\colon \Gamma^*\to\mathcal{V}_\textup{div}$ satisfies $A^{\textup{Farey}} = A^{\textup{thin}}\circ\Phi$.
\end{theorem}

\begin{proof}
    Similar to the proof of Lemma~\ref{lemma: Geometric_Meaning_of_Farey_Parameters}, we prove by induction on the number of blow-ups. That is, pick $\pi\in\mathfrak{B}$ so that $\pi = \pi'\circ \varphi$, where $\pi'\colon Y\to \operatorname{Spec}R$ is a composition of $n$ blow-ups at closed points, and $\varphi\colon Y\to Y'$ is a single blow-up with the exceptional divisor $E$ at a closed point $p\in(\pi')^{-1}(\mathfrak{m})$. The base case $n=0$ is clear. There are two cases in the inductive step, and we will frequently invoke  Lemma~\ref{lemma:Multipliticy_is_Constant_on_Segments} and Lemma~\ref{lemma:Farey_Multiplicity_Equals_Usual_Multiplicity} without further mention.

    Assume that $p$ only belongs to a single exceptional divisor $E'$ in $Y'$, so $b(E) = b(E')$. Because of Lemma~\ref{lemma:Multipliticy_is_Constant_on_Segments}, the Farey multiplicity is constant on the segment $(E', E]$ with value $m^{\textup{Farey}}(E) = b(E)$. Pick $f\in\mathfrak{m}$ so that $V(f)>E$ and $m(v_E) = \operatorname{ord}_\mathfrak{m}(f)$. Hence, $\pi^{-1}_*(f)$ must intersect $E$ transversely, and we can see that $\operatorname{ord}_E(f) = \operatorname{ord}_{E'}(f) + 1$ by a local computation. Using Proposition~\ref{prop:Using_Skewness_to_Complute_Semivaluation}, we then compute
    \begin{align*}
        A^{\textup{thin}}(v_E) - A^{\textup{thin}}(v_{E'})
        &= m(v_E)\left(\alpha(v_E) - \alpha(v_{E'})\right) \\[0.5em]
        &= m(v_E)\left(\frac{v_E(f)}{\operatorname{ord}_\mathfrak{m}(f)} - \frac{v_{E'}(f)}{\operatorname{ord}_\mathfrak{m}(f)}\right) \\[0.5em]
        &= \frac{\operatorname{ord}_E(f)}{b(E)} - \frac{\operatorname{ord}_{E'}(f)}{b(E')} \\[0.5em]
        &= \frac{1}{b(E')} \\[0.5em]
        &= \frac{a(E')+1}{b(E')} - \frac{a(E')}{b(E')} \\[0.5em]
        &= \frac{a(E)}{b(E)} - \frac{a(E')}{b(E')} \\[0.5em]
        &= A^{\textup{Farey}}(E) - A^{\textup{Farey}}(E').
    \end{align*}
    As $A^{\textup{thin}}(v_{E'}) = A^{\textup{Farey}}(E')$ from the induction hypothesis, we see that $A^{\textup{thin}}(v_E) = A^{\textup{Farey}}(E)$.

    Now, assume that $p$ lies in the intersection of two exceptional components $E'$ and $E''$ in $Y'$, where $E'<E''$ in $\Gamma^*$. As a corollary of Lemma~\ref{lemma:Multipliticy_is_Constant_on_Segments}, the Farey multiplicity is constant on the segment $(E', E'']$ with the value $m^{\textup{Farey}}(E'') = m^{\textup{Farey}}(E)$.
    Let $f\in\mathfrak{m}$ such that $m(v_{E''}) = \operatorname{ord}_\mathfrak{m}(f)$.
    We write $(a,b)$ (resp.\ $(a',b')$ and $(a'',b'')$) for the Farey parameter of $E$ (resp.\ $E'$ and $E''$).
    Leveraging Proposition~\ref{prop:Using_Skewness_to_Complute_Semivaluation} again, we obtain
    \begin{align*}
        A^{\operatorname{thin}}(v_E) - A^{\operatorname{thin}}(v_{E'})
        &= m(v_E)(\alpha(v_E)-\alpha(v_{E'})) \\[0.5em]
        &= m(v_E)\left(\frac{v_E(f)}{\operatorname{ord}_\mathfrak{m}(f)} - \frac{v_{E'}(f)}{\operatorname{ord}_\mathfrak{m}(f)}\right) \\[0.5em]
        &= \frac{\operatorname{ord}_E(f)}{b} - \frac{\operatorname{ord}_{E'}(f)}{b'} \\[0.5em]
        &= \frac{\operatorname{ord}_{E'}(f)+\operatorname{ord}_{E''}(f)}{b'+b''} - \frac{\operatorname{ord}_{E'}(f)}{b'} \\[0.5em]
        &= \frac{b'\operatorname{ord}_{E''}(f) - b''\operatorname{ord}_{E'}(f)}{b'(b'+b'')} \\[0.5em]
        &= \frac{b''}{b'+b''}\left(\frac{\operatorname{ord}_{E''}(f)}{b''}-\frac{\operatorname{ord}_{E'}(f)}{b'}\right) \\[0.5em]
        &= \frac{b''}{b'+b''}\left(A^{\textup{thin}}(v_{E''})-A^{\textup{thin}}(v_{E'})\right).
    \end{align*}
    On the other hand, we see that
    \begin{align*}
        A^{\textup{Farey}}(E) - A^{\textup{Farey}}(E')
        &= \frac{a}{b} - \frac{a'}{b'} \\[0.5em]
        &= \frac{a'+a''}{b'+b''} - \frac{a'}{b'} \\[0.5em]
        &= \frac{a''b'-a'b''}{b'(b'+b'')} \\[0.5em]
        &= \frac{b''}{b'+b''}\left(\frac{a''}{b''}-\frac{a'}{b'}\right) \\[0.5em]
        &= \frac{b''}{b'+b''}\left(A^{\textup{Farey}}(E'')-A^{\textup{Farey}}(E')\right).
    \end{align*}
    The induction hypothesis on the number of blow-ups says that $A^{\textup{thin}}(v_{E'}) = A^{\textup{Farey}}(E')$ and $A^{\textup{thin}}(v_{E''}) = A^{\textup{Farey}}(E'')$, so we get $A^{\textup{thin}}(v_{E}) = A^{\textup{Farey}}(E)$ as desired.
\end{proof}

\begin{proof}[Proof of Theorem~\ref{thm: Isom_between_Gamma_and_Vm}]
    Lemma~\ref{lemma: Isom_between_Gamma_*_and_V_div}, and Theorem~\ref{thm: A_Farey_Equals_A_Thinness} imply that the map $\Phi\colon(\Gamma^*,A^{\textup{Farey}})\to(\mathcal{V}_\textup{div},A^{\textup{thin}})$ is an isomorphism of para\-metrized non-metric $\mathbb{Q}$-trees and preserves multiplicities. Therefore, $\Phi$ extends to an isomorphism $\Phi\colon (\Gamma, A^{\textup{Farey}})\to (\mathcal{V}_\mathfrak{m}, A^{\textup{thin}})$ of parametrized non-metric $\mathbb{R}$-trees (we still write the extension of the Farey parameter as $A^{\textup{Farey}}$).

    Note that $m^{\textup{Farey}}$ and $m$ are integral-valued, and lower-semicontinuous on $\Gamma$ and on $\mathcal{V}_\mathfrak{m}$ respectively, and that $m^{\textup{Farey}} = m\circ \Phi$ on $\Gamma^*$. It follows that $m^{\textup{Farey}} = m\circ \Phi$ on $\Gamma$ as desired.
\end{proof}


\section{Newton Polyhedra of Formal Power Series}
\label{section: Newton_Polytope}

\begin{definition}[Newton Polyhedra]
Let $A = \Bbbk[x_1,x_2,\dots,x_n]$. Recall the notation $x^\mathbf{u} = x_1^{u_1}x_2^{u_2}\dots x_n^{u_n}$, where $\textbf{u}=(u_1,u_2,\dots,u_n)\in \mathbb{Z}^n_{\geq 0}$.
\begin{itemize}
    \item For $f=\sum_\mathbf{u}c_\mathbf{u}x^\mathbf{u}\in A$, the \emph{support} of $f$ is the set $\textup{Supp}(f) = \left\{ u\in\mathbb{Z}^n_{\geq 0}\,|\, c_\mathbf{u}\neq 0 \right\}$.
    \item For $f\in A$, the \emph{monomial ideal of $f$} is the ideal 
    \[
    \textup{Monom}(f) = \left( x^\mathbf{u} \mid \mathbf{u}\in\textup{Supp}(f) \right)
    \]
    of $A$. That is, it is the ideal generated by the monomials appearing in $f$.
    \item If $\mathfrak{a}$ is a monomial ideal, which means that $\mathfrak{a}$ can be generated by monomials, then the \emph{Newton polyhedron of $\mathfrak{a}$}, denoted by $\operatorname{Newt}(\mathfrak{a})$, is defined to be the convex hull of the subset $\left\{\mathbf{u}\in\mathbb{Z}^n_{\geq 0}\mid x^\mathbf{u}\in\mathfrak{a}\right\}$ of $\mathbb{R}^n_{\geq 0}$.
    \item For $f\in A$, we define the \emph{Newton polyhedron of $f$} to be $\operatorname{Newt}(\operatorname{Monom}(f))$.
\end{itemize}
The above definitions remain valid if $A = \Bbbk[[x_1,x_2,\dots,x_n]]$.
\end{definition}

\begin{theorem}
    \label{thm:Howald2001}
    If $\mathfrak{a}\subseteq\mathbb{\Bbbk}[x_1,x_2,\cdots,x_n]$ is a monomial ideal,
    then
    \[
    \operatorname{lct}_\mathfrak{m}(\mathfrak{a}) = \max\{ \lambda\in\mathbb{R}_{\geq 0}\mid (1,1,\dots,1)\in \lambda\cdot\operatorname{Newt}(\mathfrak{a}) \}.
    \]
    The same result holds if $\mathfrak{a}\subseteq\mathbb{\Bbbk}[[x_1,x_2,\cdots,x_n]]$ is a monomial ideal.
\end{theorem}

\begin{remark}
    In the original paper \cite{Howald2001}, Theorem \ref{thm:Howald2001} was proven only for the case $\mathfrak{a}\subseteq\mathbb{\mathbb{C}}[x_1,x_2,\cdots,x_n]$. In \cite{Mustata2012}, a proof involving toric resolution is given, and this method works for any base field $\Bbbk$. On the other hand, the monomial ideals in ${\Bbbk}[x_1,x_2,\cdots,x_n]$ and the ones in ${\Bbbk}[[x_1,x_2,\cdots,x_n]]$ are in canonical bijection, so the result extends to the setting of formal power series.
\end{remark}


\begin{remark}[Normalized SKPs]
    Given an SKP $[(U_j)_0^k;(\tilde{\beta})_0^k]$, if $\tilde{\beta}_0 >\tilde{\beta_1}$ then we can exchange $x$ and $y$ and assume that $\tilde{\beta}_0 \leq \tilde{\beta_1}$. Next, if $U_2$ is of the form $y-\theta_1 x$ for some $\theta_1\in \Bbbk^*$, then $\Bbbk[[x,y]] = \Bbbk[[x,U_2]]$ and thus we can replace $y$ by $U_2$. After having done so, the new SKP will have $U_2$ of the form $y^{n_1}-\theta_1x^{m_{1,0}}$ for some $\theta_1\in\Bbbk^*$ and some $m_{1,0}>1$.

    Therefore, we say that an SKP $[(U_j)_0^k;(\tilde{\beta})_0^k]$ is \emph{normalized} if $1=\tilde{\beta}_0 < \tilde{\beta_1}$.
\end{remark}

\begin{lemma}
    \label{lemma:Monomial_Valuation_of_U_j}
    Let $[(U_j)_0^k;(\tilde{\beta}_j)_0^k]$ be a normalized SKP with $k\geq 2$, and write $v_{1,\tilde{\beta}_1}$ for the monomial valuation sending $x$ to 1 and $y$ to $\tilde{\beta}_1$. For $j\geq 2$, we have
    \[
    v_{1,\tilde{\beta}_1}(U_0^{m_{j,0}}\cdots U_{j-1}^{m_{j,j-1}}) > \tilde{\beta}_1 d_{j+1},
    \]
    where $d_{j+1} \coloneq \deg_yU_{j+1} = n_1n_2\cdots n_{j-1}n_j$.\footnote{If $k< \infty$ and $j=k$, then even thought $U_{k+1}$ has not been determined yet, its degree $\deg_y U_{k+1}$ must be $n_1n_2\cdots n_{k-1}n_k$. As a result, the notation $d_{k+1}$ still makes sense.} 
\end{lemma}

\begin{proof}
    Recall the skewness $\alpha$ in Subsection~\ref{subsection: Parametrizations_on_V}. By Proposition~\ref{prop:Using_Skewness_to_Complute_Semivaluation}, for $l>0$ we see that
    \[
    v_{1,\tilde{\beta}_1}(U_l) = \alpha(v_{1,\tilde{\beta}_1}\wedge v_{U_l})v_\mathfrak{m}(U_l) = \alpha(v_{1,\tilde\beta_1})\deg_yU_l = \tilde\beta_1 d_l.
    \]
    Now for $j\geq 2$, we can compute
    \begin{align*}
    v_{1,\tilde{\beta}_1}(U_0^{m_{j,0}}\cdots U_{j-1}^{m_{j,j-1}})
    &= \sum_{l=0}^{j-1}m_{j,l}\,v_{1,\tilde{\beta}_1}(U_l)\\
    &= m_{j,0} + m_{j,1}\tilde\beta_1 d_1 + m_{j,2}\tilde\beta_1 d_2 + \cdots + m_{j,j-1}\tilde\beta_1 d_{j-1}.
    \end{align*}
    Therefore, to prove the lemma, it suffices to solve for $(m_{j,0},m_{j,1},\dots,m_{j,j-1})\in\mathbb{R}^j_{\geq 0}$ in the following optimization problem
    \begin{align*}
    \begin{array}{ll}
    \min &\displaystyle{m_{j_0}+\tilde{\beta}_1\sum_{l=1}^{j-1} m_{j,l} d_l} \\
    \textup{subject to}
    &n_j\tilde{\beta}_j = \displaystyle{\sum_{l=0}^{j-1}m_{j,l}\tilde{\beta}_l} \\
    &m_{j,l} \geq 0 \,\mbox{ for }\, l=0,\dots,j-1,
    \end{array}
    \end{align*}
    where the constraint is from \ref{(S2):SKP_relations} of the definition of an SKP, and to show that the minimum is still greater than $\tilde{\beta}_1 d_{j+1}$.

    This is a linear programming problem, so we know that the minimum is on a vertex of the simplex determined by the constrains. Each such vertex is of the form
    \[
    (m_{j,0},m_{j,1},\dots,m_{j,j-1}) = n_j \tilde{\beta}_j\tilde{\beta}_l^{-1} \textbf{e}_l,
    \]
    where $\textbf{e}_l\in\mathbb{R}^j$ is the $l$-th standard basis vector. In view of Lemma~\ref{lemma: First_Prop_of_SKP}, the sequence $(\tilde{\beta}_l/d_l)$ is strictly increasing, so the vertex $(0,0,\dots,0,n_j\tilde{\beta}_j\tilde{\beta}_{j-1}^{-1})$ gives the minimum $\tilde{\beta}_1 n_j\tilde{\beta}_j\tilde{\beta}_{j-1}^{-1}d_{j-1}$. Using the relation $\tilde{\beta}_j > n_{j-1}\tilde{\beta}_{j-1}$, we see the desired inequality
    \[
    \tilde{\beta}_1 n_j\tilde{\beta}_j\tilde{\beta}_{j-1}^{-1}d_{j-1}
    > \tilde{\beta}_1 n_jn_{j-1}d_{j-1} = \tilde{\beta}_1 d_{j+1}.
    \]
\end{proof}

\begin{theorem}
    \label{thm:Newton_Polyhedron_of_U_j}
    Let $[(U_j)_0^k;(\tilde{\beta}_j)_0^k]$ be an SKP with $2\leq k\leq \infty$ such that $\tilde{\beta}_0 < \tilde{\beta}_1$.
    For all $2\leq j \leq k$, we have the following equality
    \[
    \operatorname{Newt}(U_j) = \frac{\deg_y U_j}{n_1}\cdot \operatorname{Newt}(U_2).
    \]
    Here, $\deg_y U_\infty$ only makes sense when $U_\infty$ is defined.
\end{theorem}

\begin{proof}
    Without loss of generality, we may assume $\tilde{\beta}_0=1$ to make the SKP normalized. We write $d_j=\deg_y U_j$, and prove the theorem by induction on $j<\infty$; it is true for $j=2$ because $d_2 = n_1$, so we assume that $j>2$ and that the theorem is true for smaller $j$'s.

    Write $U_0^{m_{j,0}}\cdots U_{j-1}^{m_{j,j-1}}=\sum c_{ij}x^iy^j$. Lemma~\ref{lemma:Monomial_Valuation_of_U_j} indicates that
    \[
    i+\tilde{\beta}_1j = v_{1,\tilde{\beta}_1}(x^iy^j) \geq v_{1,\tilde{\beta}_1}\left( \sum c_{ij}x^iy^j \right) > \tilde{\beta}_1 d_{j+1},
    \]
    which is equivalent to
    \[
    \frac{n_1}{d_{j+1}}\left(\frac{i}{m_{1,0}} + \frac{j}{n_1}\right) > 1
    \]
    due to the relation $n_1\tilde{\beta}_1 = m_{1,0}\tilde{\beta}_0 = m_{1,0}$.
    Since
    \begin{align}
    \label{equation:Discription_of_Newt(U_2)}
    \operatorname{Newt}(U_2) = \left\{ (u_1,u_2)\in\mathbb{R}^2_{\geq 0}\,\middle|\, \frac{u_1}{m_{1,0}}+\frac{u_2}{n_1}\geq 1\right\},
    \end{align}
    the indicated inequality implies that the point $(i,j)\in\mathbb{R}^2_{\geq 0}$ lies in the interior of the set $(d_{j+1}/n_1)\cdot\operatorname{Newt}(U_2)$.
    By the induction hypothesis, we also have
    \[
    \operatorname{Newt}(U_j^{n_j}) = n_j\cdot\operatorname{Newt}(U_j) = n_j\cdot\frac{d_j}{n_1}\cdot\operatorname{Newt(U_2)} = \frac{d_{j+1}}{n_1}\cdot\operatorname{Newt}(U_2),
    \]
    so we can conclude that 
    \begin{align*}
    \operatorname{Newt}(U_{j+1})
    &= \operatorname{Newt}(U_j^{n_j}-\theta_jU_0^{m_{j,0}}\cdots U_{j-1}^{m_{j,j-1}}) \\
    &\subseteq \frac{d_{j+1}}{n_1}\cdot\operatorname{Newt}(U_2) \\
    &= \operatorname{Newt}(U_j^{n_j}) \\
    &\subseteq \operatorname{Newt}(U_{j+1}).
    \end{align*}
    Finally, assume $j = \infty$ in the case that $U_\infty$ is defined. Since $U_j\to U_\infty$ and $d_\infty=d_{j_0}$ for some $j_0$ large enough, we see that 
    \[
    \operatorname{Newt}(U_\infty) = \operatorname{Newt}(U_{j_0})
    = \frac{d_{j_{j_0}}}{n_1}\cdot\operatorname{Newt}(U_2) 
    = \frac{d_{\infty}}{n_1}\cdot\operatorname{Newt}(U_2)
    \]
    as desired.
\end{proof}

\begin{remark}
    \label{remark:Newton_Polyhedron_of_General_U_j}
    We have seen that (\ref{equation:Discription_of_Newt(U_2)}) gives an explicit description of $\operatorname{Newt}(U_2)$, so combining Theorem~\ref{thm:Newton_Polyhedron_of_U_j} we deduce that, for $j\geq 2$,
    \[
    \operatorname{Newt}(U_j) = \left\{ (u_1,u_2)\in\mathbb{R}^2_{\geq 0}\,\middle|\, \frac{n_1}{\deg_yU_j}\left(\frac{u_1}{m_{1,0}}+\frac{u_2}{n_1}\right)\geq 1\right\}.
    \]
    In view of the relation $n_1\tilde{\beta}_1 = m_{1,0}\tilde{\beta}_0$ and Theorem~\ref{thm:Howald2001}, we can conclude that
    \[
    \operatorname{lct}_o\left(\operatorname{Monom}(U_k)\right) = \frac{n_1}{\deg_y U_k}\left(\frac{1}{n_1}+\frac{1}{m_{1,0}}\right)
    =\frac{1}{\deg_y U_k}\left( 1 + \frac{\tilde{\beta}_0}{\tilde{\beta}_1} \right).
    \]
\end{remark}


\section{The Main Theorem}
\label{section: Main_Theorem}
\begin{theorem}
\label{thm:LCT_Formula}
Let ${f}\in\Bbbk[[x,y]]$ be irreducible, and assume that the tangent cone of $V({f})$ is defined by the equation $y^{\operatorname{ord}_\mathfrak{m}({f})} = 0$. Write $v_{{f}}$ for the corresponding non-normalized curve semivaluation. Then, we have
\[
\operatorname{lct}_\mathfrak{m}({f}) = \frac{1}{v_{{f}}(x)}+\frac{1}{v_{f}(y)}.
\]
\end{theorem}

\begin{proof}
Under the assumption on the tangent cone of ${f}$, we can write $v_{f}= \operatorname{val}[(U_j)_0^k;(\tilde{\beta})_j^k]$ so that ${f} = U_k$, $1\leq k\leq \infty$, and $\tilde{\beta}_0<\tilde{\beta}_1$. We denote $\deg_y U_j$ by $d_j$ as usual. To prove the desired equality, we verify the $(\leq)$ and $(\geq)$ directions.

For the $(\leq)$ direction, we apply Theorem~\ref{thm:Newton_Polyhedron_of_U_j} to get
\begin{align*}
\operatorname{lct}_\mathfrak{m}({f})
&= \operatorname{lct}_\mathfrak{m}(U_k) \\
&\leq \operatorname{lct}_\mathfrak{m}(\textup{Monom}(U_k)) \\
&= \frac{1}{d_k}\left( 1 + \frac{\tilde{\beta}_0}{\tilde{\beta}_1} \right) \\
&= \frac{1}{v_{{f}}(x)}+\frac{1}{v_{f}(y)},
\end{align*}
where the inequality is from the fact that $\mathfrak{a}\subseteq\mathfrak{b}$ implies $\operatorname{lct}_\mathfrak{m}(\mathfrak{a})\leq\operatorname{lct}_\mathfrak{m}(\mathfrak{b})$, and the last equality is because $\tilde{\beta}_0 = v_{{f}}(x) = v_x({f}) = \alpha(v_x\wedge v_{f})\deg_y {f} = d_k$.

For the $(\geq)$ direction, let $\pi: Y\to X\coloneq\operatorname{Spec}\Bbbk[[x,y]]$ be a log resolution of $V({f})$, and pick an exceptional component $E\subseteq \pi^{-1}(\mathfrak{m})$. Our goal is to show that
\[
\frac{A_X(E)}{\operatorname{ord}_E(\pi^*{f})}\geq
\frac{1}{v_{{f}}(x)}+\frac{1}{v_{f}(y)}.
\]
We have the following estimate:
\begin{align*}
    \frac{A_X(E)}{\operatorname{ord}_E(\pi^*{f})}
    &= \frac{\operatorname{ord}_E(\pi^*\mathfrak{m})}{\operatorname{ord}_E(\pi^*{f})} \cdot \frac{A_X(E)}{\operatorname{ord}_E(\pi^*\mathfrak{m})}\\
    &= \frac{1}{v_E({f})} \cdot A^{\textup{thin}}(E)\\
    &= \frac{1}{v_\mathfrak{m}({f})\alpha(v_E\wedge v_{f})}\cdot A^{\textup{thin}}(v_E) \\
    &\geq \frac{1}{d_k}\cdot\frac{1}{\alpha( v_E\wedge v_{f})}\cdot A^{\textup{thin}}( v_E\wedge  v_{f}),
\end{align*}
where the second equality is from Theorem~\ref{thm: Isom_between_Gamma_and_Vm}. Hence, it remains to show the following inequality
\[
\frac{A^{\textup{thin}}(v_E\wedge v_{f})}{\alpha( v_E\wedge v_{f})} \geq  1 + \frac{\tilde{\beta}_0}{\tilde{\beta}_1}.
\]
Notice that $A^{\textup{thin}}/\alpha$ defines a real-valued function on the valuative tree $\mathcal{V}$. We claim that $A^{\textup{thin}}/\alpha$ satisfies the following lemma and postpone its proof to the end of this section.

\begin{lemma}
\label{lemma:A_over_alpha}
Let $ v\in\mathcal{V}$ with $m(v)<\infty$ (so $v$ is not infinitely singular), and let $ v_1$ be the first divisorial valuation after $ v_\mathfrak{m}$ in the approximating sequence of $ v$. Then $A^{\textup{thin}}/\alpha$ is strictly decreasing on $[ v_\mathfrak{m}, v_1]$ and strictly increasing on $[ v_1, v]$. Moreover, the minimum is $A^{\textup{thin}}( v_1)/\alpha( v_1) = 1 + 1/\alpha( v_1)$.
\end{lemma}

Setting $ v =  v_{f}$ in Lemma~\ref{lemma:A_over_alpha}, we now show that $ v_1 =  v_y\wedge v_{f}$. On $[v_\mathfrak{m}, v_y\wedge v_{f}]$, the multiplicity function $m$ is 1, so $ v_1\geq  v_y\wedge v_{f}$. Conversely, consider a valuation $ w \in ( v_y\wedge v_{f},  v_{f})$. Because $ w$ is not monomial, we must have $m(w)>1$; therefore, $ v_1\leq  v_y\wedge v_{f}$.

Since $ v_E\wedge v_{f}\in [ v_\mathfrak{m}, v_{f})$, by Lemma \ref{lemma:A_over_alpha} we obtain
\begin{align*}
\frac{A^{\textup{thin}}( v_E\wedge v_{f})}{\alpha( v_E\wedge v_{f})}
\geq \frac{A^{\textup{thin}}( v_y\wedge v_{f})}{\alpha( v_y\wedge v_{f})}
= 1+\frac{1}{\alpha( v_y\wedge v_{f})}
= 1+\frac{\tilde{\beta}_0}{\tilde{\beta}_1}.
\end{align*}
This completes the proof of the theorem.
\end{proof}

\begin{proof}[Proof of Lemma~\ref{lemma:A_over_alpha}]
Let $[(v_j)_0^g;(m_j)_0^g]$, $0\leq g < \infty$, be the approximating sequence of $v$. Recall from the definition of approximating sequence that $[(v_j)_0^g;(m_j)_0^g]$ consists of an increasing sequence of semivaluations
\[ 
v_\mathfrak{m} =  v_0 <  v_1 <  v_2 < \cdots <  v_g <  v_{g+1} =  v
\]
and an increasing sequence of positive integers $1 = m_1 < m_2 < \cdots < m_g < m_{g+1}$ such that $m( w) = m_j$ for $ w\in( v_{j-1}, v_j]$. Put $t_j = \alpha( v_j)$ for $j = 0,\dots,g+1$.

For $ w \in [v_\mathfrak{m},v_1]$, we see that the function
\[
\frac{A^{\textup{thin}}( w)}{\alpha( w)} = \frac{2 + 1\cdot(\alpha( w)-t_0)}{\alpha( w)}
= 1+\frac{1}{\alpha( w)}
\]
is strictly decreasing as $ w$ increases. Furthermore, $A( v_1)/\alpha( v_1) = 1+1/\alpha( v_1)$.

On the other hand, let $ w\in( v_{j-1}, v_j]$ for $j\geq 2$.
Because $m_l-m_{l+1} \leq -1$ (from Theorem~\ref{prop: SKP_Computes_Skewness_and_Multiplicity}, $m_l$ divides $m_{l+1}$ non-trivially) and $t_l>1$ for any $l>0$, the following computation
\begin{align*}
\frac{A^{\textup{thin}}( w)}{\alpha( w)}
&= \frac{1}{\alpha( w)}\left( 2 + \sum_{l=1}^{j-1}m_l(t_l-t_{l-1})+m_j(\alpha( w)-t_{j-1})\right) \\
&= m_j + \frac{1}{\alpha( w)}\left(1+\sum_{l=1}^{j-1}t_l(m_l-m_{l+1})\right)
\end{align*}
implies that $A^{\textup{thin}}/\alpha$ is strictly increasing on $( v_{j-1}, v_j]$. 
Since both $\alpha$ and $A^{\textup{thin}}$ are parameterizations, they are continuous (for the observer's topology) on $[v_{\mathfrak{m}}, v]$, and thus $A^{\textup{thin}}/\alpha$ is continuous on $[v_{\mathfrak{m}}, v]$ too, which shows that $A^{\textup{thin}}/\alpha$ is strictly increasing on $[v_1, v]$.
\end{proof}


\bibliographystyle{alpha}
\bibliography{reference2}

@misc{stacks-project,
  author       = {The {Stacks project authors}},
  title        = {The Stacks project},
  howpublished = {\url{https://stacks.math.columbia.edu}},
  year         = {2025},
}

@incollection {Mustata2012,
    AUTHOR = {Musta\c{t}\u{a}, Mircea},
     TITLE = {I{MPANGA} lecture notes on log canonical thresholds},
 BOOKTITLE = {Contributions to algebraic geometry},
    SERIES = {EMS Ser. Congr. Rep.},
     PAGES = {407--442},
      NOTE = {Notes by Tomasz Szemberg},
 PUBLISHER = {Eur. Math. Soc., Z\"urich},
      YEAR = {2012},
      ISBN = {978-3-03719-114-9},
   MRCLASS = {14B05 (13A35 14E30)},
  MRNUMBER = {2976952},
MRREVIEWER = {Ali\ Sinan\ Sert\"oz},
       DOI = {10.4171/114-1/16},
       URL = {https://doi.org/10.4171/114-1/16},
}

@incollection {deFernex&Ein&Mustata2011,
    AUTHOR = {de Fernex, Tommaso and Ein, Lawrence and Musta\c{t}\u{a},
              Mircea},
     TITLE = {Log canonical thresholds on varieties with bounded
              singularities},
 BOOKTITLE = {Classification of algebraic varieties},
    SERIES = {EMS Ser. Congr. Rep.},
     PAGES = {221--257},
 PUBLISHER = {Eur. Math. Soc., Z\"urich},
      YEAR = {2011},
      ISBN = {978-3-03719-007-4},
   MRCLASS = {14E15 (14B05 14E30)},
  MRNUMBER = {2779474},
MRREVIEWER = {Ali\ Sinan\ Sert\"oz},
       DOI = {10.4171/007-1/10},
       URL = {https://doi.org/10.4171/007-1/10},
}

@article {Coulbois2007,
    AUTHOR = {Coulbois, Thierry and Hilion, Arnaud and Lustig, Martin},
     TITLE = {Non-unique ergodicity, observers' topology and the dual
              algebraic lamination for {$\mathbb R$}-trees},
   JOURNAL = {Illinois J. Math.},
  FJOURNAL = {Illinois Journal of Mathematics},
    VOLUME = {51},
      YEAR = {2007},
    NUMBER = {3},
     PAGES = {897--911},
      ISSN = {0019-2082,1945-6581},
   MRCLASS = {20E08 (20E05 20F65 20F67 37A99 57M07)},
  MRNUMBER = {2379729},
MRREVIEWER = {Vincent\ Guirardel},
       URL = {http://projecteuclid.org/euclid.ijm/1258131109},
}

@book {SwansonHuneke2006,
    AUTHOR = {Huneke, Craig and Swanson, Irena},
     TITLE = {Integral closure of ideals, rings, and modules},
    SERIES = {London Mathematical Society Lecture Note Series},
    VOLUME = {336},
 PUBLISHER = {Cambridge University Press, Cambridge},
      YEAR = {2006},
     PAGES = {xiv+431},
      ISBN = {978-0-521-68860-4; 0-521-68860-4},
   MRCLASS = {13B22 (13A18 13A30 13A35 13H15 14A05)},
  MRNUMBER = {2266432},
MRREVIEWER = {Liam\ O'Carroll},
}

@book{FavreJonsson2004,
    AUTHOR = {Favre, Charles and Jonsson, Mattias},
     TITLE = {The valuative tree},
    SERIES = {Lecture Notes in Mathematics},
    VOLUME = {1853},
 PUBLISHER = {Springer-Verlag, Berlin},
      YEAR = {2004},
     PAGES = {xiv+234},
      ISBN = {3-540-22984-1},
   MRCLASS = {13A18 (12J20 13F30 14H20 37F99)},
  MRNUMBER = {2097722},
MRREVIEWER = {Karen\ E.\ Smith},
       DOI = {10.1007/b100262},
       URL = {https://doi.org/10.1007/b100262},
      NOTE = {The online version includes an errata.}
}

@article {Howald2001,
    AUTHOR = {Howald, J. A.},
     TITLE = {Multiplier ideals of monomial ideals},
   JOURNAL = {Trans. Amer. Math. Soc.},
  FJOURNAL = {Transactions of the American Mathematical Society},
    VOLUME = {353},
      YEAR = {2001},
    NUMBER = {7},
     PAGES = {2665--2671},
      ISSN = {0002-9947,1088-6850},
   MRCLASS = {14M25 (14E05 14Q99)},
  MRNUMBER = {1828466},
MRREVIEWER = {Shihoko\ Ishii},
       DOI = {10.1090/S0002-9947-01-02720-9},
       URL = {https://doi.org/10.1090/S0002-9947-01-02720-9},
}

@incollection {Kollar1997,
    AUTHOR = {Koll\'ar, J\'anos},
     TITLE = {Singularities of pairs},
 BOOKTITLE = {Algebraic geometry---{S}anta {C}ruz 1995},
    SERIES = {Proc. Sympos. Pure Math.},
    VOLUME = {62, Part 1},
     PAGES = {221--287},
 PUBLISHER = {Amer. Math. Soc., Providence, RI},
      YEAR = {1997},
      ISBN = {0-8218-0894-X; 0-8218-0493-6},
   MRCLASS = {14E30 (14E15 32S40)},
  MRNUMBER = {1492525},
MRREVIEWER = {Alessio\ Corti},
       DOI = {10.1090/pspum/062.1/1492525},
       URL = {https://doi.org/10.1090/pspum/062.1/1492525},
}

@incollection {Igusa1977,
    AUTHOR = {Igusa, J.},
     TITLE = {On the first terms of certain asymptotic expansions},
 BOOKTITLE = {Complex analysis and algebraic geometry},
     PAGES = {357--368},
 PUBLISHER = {Iwanami Shoten Publishers, Tokyo},
      YEAR = {1977},
   MRCLASS = {14G20 (10H25 12B30)},
  MRNUMBER = {485881},
MRREVIEWER = {Stephen\ Haris},
}

@article {MacLane1936,
    AUTHOR = {MacLane, Saunders},
     TITLE = {A construction for absolute values in polynomial rings},
   JOURNAL = {Trans. Amer. Math. Soc.},
  FJOURNAL = {Transactions of the American Mathematical Society},
    VOLUME = {40},
      YEAR = {1936},
    NUMBER = {3},
     PAGES = {363--395},
      ISSN = {0002-9947,1088-6850},
   MRCLASS = {13A18 (13F20)},
  MRNUMBER = {1501879},
       DOI = {10.2307/1989629},
       URL = {https://doi.org/10.2307/1989629},
}

\end{document}